\documentclass[11pt,reqno]{amsart}
\usepackage{amssymb,amsmath,tabularx}
\usepackage{amsthm,verbatim}
\usepackage[all]{xy}
\usepackage[bookmarks=true]{hyperref}
\numberwithin{equation}{section}

\newtheorem{thm}{Theorem}[section]
\newtheorem{lem}[thm]{Lemma}

\newtheorem{cor}[thm]{Corollary}
\theoremstyle{definition}

\newcommand{\module}{\operatorname{\mathrm{mod}}}
\newcommand{\GL}{\mathrm{GL}}

\newcommand{\Lie}{\operatorname{\mathrm{Lie}}}

\newcommand{\Liemax}{\operatorname{\mathrm{Lie}^\mathrm{max}}}
\newcommand{\sym}[1]{\mathfrak{S}_{#1}}
\newcommand{\Ind}{\operatorname{Ind}}

\begin{document}
\title[Asymptotic behaviour]{Asymptotic behaviour of \\ Lie powers and Lie modules}

\author{Roger M. Bryant}
\address[Roger M. Bryant]{School of Mathematics, University of Manchester, Manchester M13 9PL, UK.}
\email{roger.bryant@manchester.ac.uk}

\author{Kay Jin Lim}
\author{Kai Meng Tan}
\address[Kay Jin Lim and Kai Meng Tan]{Department of Mathematics, National University of Singapore, Block S17, 10 Lower Kent Ridge Road, Singapore 119076.}
\email[K. J. Lim]{matlkj@nus.edu.sg}
\email[K. M. Tan]{tankm@nus.edu.sg}

\date{November 2010}

\thanks{Supported by EPSRC Standard Research Grant EP/G024898/1 and EP/G025487/1, and by MOE Academic Research Fund R-146-000-135-112.}

\subjclass[2000]{17B01, 20C30 (primary), 20C20 (secondary)}

\begin{abstract}
Let $V$ be a finite-dimensional $FG$-module, where $F$ is a field of prime characteristic $p$ and $G$ is a group.  We show that, when $r$ is not a power of $p$, the Lie power $L^r(V)$ has a direct summand $B^r(V)$ which is a direct summand of the tensor power $V^{\otimes r}$ and which satisfies $\dim B^r(V)/\dim L^r(V) \to 1$ as $r \to \infty$.  Similarly, for the same values of $r$, we obtain a projective submodule $C(r)$ of the Lie module $\Lie(r)$ over $F$ such that $\dim C(r)/\dim \Lie(r) \to 1$ as $r \to \infty$.
\end{abstract}

\maketitle

\section{Introduction}\label{S:intro}

Let $F$ be a field of prime characteristic $p$, $G$ a group, and $V$ a finite-dimensional (right) $FG$-module.  Write $T(V)$ for the tensor algebra of $V$:  thus $T(V) = \bigoplus_{r=0}^{\infty} V^{\otimes r}$.  Let $L(V)$ denote the Lie subalgebra of $T(V)$ generated by $V$:  thus $L(V) = \bigoplus_{r=1}^{\infty} L^r(V)$, where $L^r(V) = L(V) \cap V^{\otimes r}$.  We call $L^r(V)$ the $r$th Lie power of $V$.  The tensor power $V^{\otimes r}$ is an $FG$-module (under the `diagonal' action of $G$) and $L^r(V)$ is a submodule of $V^{\otimes r}$.  The dimension of $L^r(V)$ is given by a formula of Witt (see \cite[Theorem 5.11]{WMAKDS}):
\begin{equation} \label{E:dim(L^r(V))}
\dim L^r(V) = \frac{1}{r} \sum_{d \mid r} \mu(d) n^{r/d},
\end{equation}
where $n = \dim V$ and $\mu$ denotes the M\"{o}bius function.

The modules $L^r(V)$ have been extensively studied in recent years.  We refer to the paper of Bryant and Schocker \cite{RBMS} and the works cited there for details of progress on the problem of describing $L^r(V)$ up to isomorphism.  When $r$ is not divisible by $p$, $L^r(V)$ is a direct summand of $V^{\otimes r}$ (see, for example, \cite[\S 3.1]{SDKE}).  Although this does not hold in general, our first main result shows that, in an asymptotic sense, `most' of $L^r(V)$ is a direct summand of $V^{\otimes r}$ provided that $r$ is not a power of $p$.

\begin{thm} \label{T:main1}
Suppose that $\dim V>1$. Let $\mathcal{A}$ be the set of all positive integers $r$ such that $r$ is not a power of $p$. Then, for each $r \in \mathcal{A}$, there is a direct summand $B^r(V)$ of $L^r(V)$ such that $B^r(V)$ is a direct summand of $V^{\otimes r}$ and
$$
\lim_{\substack{r \to \infty \\ r \in \mathcal{A}}} \frac{\dim B^r(V)}{\dim L^r(V)} = 1.
$$
\end{thm}

The modules $B^r(V)$ are the modules denoted by $B_r$ in the decomposition theorem of \cite{RBMS} and their dimensions are given by a recurrence formula.  When $r = p^m$ with $m > 0$, we have $B_r = 0$.  Thus, although we conjecture that a result like Theorem \ref{T:main1} holds without any restrictions on $r$, we cannot use \cite{RBMS} when $r$ is a power of $p$.

The other main result of this paper concerns the Lie module $\Lie(r)$.  This is a module for the symmetric group $\sym{r}$ and can be defined as follows.  Let $E_n$ be a vector space over $F$ of finite dimension $n$, where $n \geqslant r$, and let $\{e_1,\dotsc, e_n\}$ be a basis for $E_n$.  Then $\Lie(r)$ is the subspace of $L^r(E_n)$ spanned by all elements of the form $[e_{1\pi},e_{2\pi}, \dotsc, e_{r\pi}]$ where $\pi \in \sym{r}$ and $[e_{1\pi},e_{2\pi}, \dotsc, e_{r\pi}]$ denotes the `left-normed' Lie product in which $[e_{1\pi},e_{2\pi}] = e_{1\pi} \otimes e_{2\pi} - e_{2\pi}\otimes e_{1\pi}$ and $[e_{1\pi},e_{2\pi}, \dotsc, e_{r\pi}] = [[e_{1\pi},e_{2\pi}, \dotsc, e_{(r-1)\pi}], e_{r\pi}]$.  Clearly $\Lie(r)$ is a (right) $F\sym{r}$-module, where the action of $\sym{r}$ comes from its right action on $\{1,\dotsc, r\}$. Furthermore, the definition of $\Lie(r)$ is independent of the choice of $n$ with $n\geqslant r$. The dimension of $\Lie(r)$ is $(r-1)!$ (as follows from \cite[Theorem 5.11]{WMAKDS} by taking $n_1=n_2=\cdots=1$).

The main motivation for our result on $\Lie(r)$ comes from the work of Selick and Wu \cite{SW}.  These authors considered the problem of finding natural homotopy decompositions of the loop suspension of a $p$-torsion suspension and proved that this problem is equivalent to the algebraic problem of finding natural coalgebra decompositions of the primitively generated tensor algebras over the field with $p$ elements.  They determined (see \cite[Theorem 6.5]{SW}) the finest coalgebra decomposition of a tensor algebra (over an arbitrary field).  Their result can be described as a functorial Poincar\'e-Birkhoff-Witt theorem.  However, in order to compute the factors in this decomposition, it is necessary to know a maximal projective submodule, $\Lie^{\max}(r)$, of the Lie module $\Lie(r)$.

The projective modules for the symmetric groups over fields of
positive characteristic $p$ are not known in general.  Their structure depends
on the decomposition matrices for symmetric groups and the determination of these is a famous open problem.  However, according to \cite{SW2}, even if the modules $\Lie^{\max}(r)$
cannot be computed precisely, it is of interest to know how quickly their dimensions grow and
whether or not the growth rate is exponential. Determination of $\Liemax(6)$ and $\Liemax(8)$ in characteristic 2 in \cite{SW2} suggests that $\Liemax(r)$ is relatively
large compared with $\Lie(r)$.  If this is true in general, it has the desirable consequence that the factors in the functorial PBW theorem are relatively small.

Some progress has been made in the understanding of $\Liemax(r)$ in the case where $r = pk$ with $p \nmid k$.  Erdmann and Schocker \cite{ES} established, amongst other things, that there is a one-to-one correspondence between the non-projective indecomposable summands of $\Lie(pk)$ and the indecomposable summands of $\Lie(k)$.  Erdmann and Tan \cite{ET} provided an upper bound for $\dim \Liemax(pk)$ and showed that the ratio of this upper bound to $\dim \Lie(pk)$ approaches $1$ as $k$ tends to infinity.  They conjectured that the same fact holds for $\dim \Liemax(pk)/ \dim \Lie(pk)$.

Our second main result establishes this conjecture and considerably more.

\begin{thm} \label{T:main2}
Let $\mathcal{A}$ be the set of all positive integers $r$ such that $r$ is not a power of $p$.  Then
$$
\lim_{\substack{r\to \infty \\r \in \mathcal{A}}} \frac{\dim \Liemax(r)}{\dim\Lie(r)} = 1.
$$
\end{thm}

The theorem shows, in particular, that $\dim\Liemax(r)$ grows exponentially with $r$ when $r$ is not a power of $p$.  As in the case of Theorem \ref{T:main1}, we conjecture that the restriction on $r$ can be removed.  However, our methods work only for $r \in \mathcal{A}$.  In the proof, we quickly reduce to the case where $F$ is an infinite field and then consider the $n$-dimensional natural module $E_n$ for the general linear group $\GL_n(F)$, where $n \geqslant r$.  The image of $L^r(E_n)$ under the Schur functor is the Lie module $\Lie(r)$ (see Section \ref{lie module section}).  With $B^r(E_n)$ as in Theorem \ref{T:main1}, the image $C(r)$ of $B^r(E_n)$ under the Schur functor is a projective submodule of $\Lie(r)$.  We describe the dimension of $C(r)$ by means of a recurrence formula and show that $\dim C(r)/ \dim \Lie(r) \to 1$ as $r \to \infty$ with $r \in \mathcal{A}$.

Theorem \ref{T:main1} will be proved in Section \ref{lie power section} and Theorem \ref{T:main2} in Section \ref{lie module section}.  Throughout this paper $F$ denotes a field of prime characteristic $p$.  All of our algebras and vector spaces are taken over $F$, unless otherwise stated, and all tensor products are also taken over $F$.  All modules are right modules.

\section{Modular Lie powers} \label{lie power section}

In this section we shall prove Theorem \ref{T:main1}.

Let $G$ be a group and let $V$ be a finite-dimensional $FG$-module.  The decomposition theorem of \cite{RBMS} yields, for each positive integer $r$, a certain direct summand $B_r$ of $L^r(V)$ such that $B_r$ is a direct summand of $V^{\otimes r}$.  Here we write $B^r(V)$ instead of $B_r$. We shall use the following result, where we write $p^iU$ for the direct sum of $p^i$ copies of a module $U$.

\begin{thm}[{\cite[Theorem 4.2]{RBMJ}}]\label{decomposition theorem}
Let $k$ be a positive integer not divisible by $p$.  Then, for each non-negative
integer $m$, we have
\begin{equation} \label{E:decompthm}
p^m B^{p^mk}(V) \oplus p^{m-1} B^{p^{m-1}k}(V)^{\otimes p} \oplus \dotsb \oplus B^k(V)^{\otimes p^m} \cong L^k(V^{\otimes p^m}).
\end{equation}
\end{thm}

For positive integers $n$ and $r$, let
$$
w(n,r) = \frac{1}{r} \sum_{d \mid r} \mu(d) n^{r/d}.
$$

\begin{lem} \label{L:maininequality}
For all $n$ and $r$, we have
$$
n^r/r - n^{r/2}/2 \leqslant w(n,r) \leqslant n^r/r.
$$
\end{lem}

\begin{proof}
The result is clear if $r=1$, so we assume $r \geqslant 2$.  By \cite[Theorem 7.1]{CR}, $rw(n,r)$ is the number of words of length $r$ on an alphabet of cardinality $n$ that cannot be written as a power of a shorter word.  Clearly $rw(n,r) \leqslant n^r$.  If $u$ is a word of length $r$ of the form $v^d$, where $d > 1$, then $v$ has length at most $r/2$.  Hence, for a given $d$, the number of possibilities for $v$ is at most $n^{r/2}$.  The number of possibilities for $d$ is at most $r/2$. Therefore $n^r - rw(n,r) \leqslant (r/2)n^{r/2}$, which gives the required result.
\end{proof}

Let $n = \dim(V)$ and suppose that $n \geqslant 2$.  By \eqref{E:dim(L^r(V))}, we have $\dim L^r(V) = w(n,r)$.  For each $r \geqslant 1$, we define $$b_r = \frac{\dim B^r(V)}{\dim L^r(V)}=\frac{\dim B^r(V)}{w(n,r)}.$$  Since $B^r(V)$ is a submodule of $L^r(V)$ we have $0 \leqslant b_r \leqslant 1$.  In order to prove Theorem \ref{T:main1}, we need to prove that $b_r \to 1$ as $r \to \infty$ with $r \in \mathcal{A}$.  Note that for $r \in \mathcal{A}$ we can write $r = p^m k$ where $m \geqslant 0$, $p \nmid k$ and $k \geqslant 2$.  Suppose, for the rest of this section, that $m$ and $k$ satisfy these conditions.  By direct computations of the dimensions in \eqref{E:decompthm}, we have
\begin{equation} \label{E:eq}
\sum_{i=0}^m p^{m-i}(b_{p^{m-i}k})^{p^i} w(n,p^{m-i}k)^{p^i} = w(n^{p^m}, k).
\end{equation} For $i = 0,1, \dotsc, m$, define
$$
a_i = \frac{w(n, p^{m-i}k)^{p^i}}{p^i w(n,p^m k)}.
$$
Then, on dividing \eqref{E:eq} by $p^m w(n, p^m k)$ and re-arranging, we obtain
\begin{equation} \label{E:eq2}
b_{p^m k} = \frac{w(n^{p^m},k)}{p^m w(n, p^m k)} - \sum_{i=1}^m a_i (b_{p^{m-i}k})^{p^i}.
\end{equation}

\begin{lem} \label{L:<=1}
Suppose that $0 < s \leqslant i \leqslant m$.  Then
$$
\frac{a_i}{a_{i-s}} \leqslant p^{-s} \left(\frac{2p^s}{(p^{m-i}k)^{p^s -1}} \right)^{p^{i-s}}.
$$
\end{lem}

\begin{proof}
We have
$$
 \frac{a_i}{a_{i-s}} = \frac{w(n, p^{m-i}k)^{p^i}}{p^s w(n, p^{m-i+s} k)^{p^{i-s}}} = p^{-s} \left( \frac{w(n, p^{m-i}k)^{p^s}}{w(n, p^{m-i+s} k)} \right)^{p^{i-s}}.
$$
Write $q = p^{m-i+s}k$.  Then, by Lemma \ref{L:maininequality},
$$
\frac{w(n,p^{m-i}k)^{p^s}}{w(n,p^{m-i+s}k)} \leqslant
\frac{(p^{m-i}k)^{-p^s}n^q}{ n^q/q - n^{q/2}/2} =
\frac{(p^{m-i}k)^{-p^s}qn^{q/2}}{n^{q/2} - q/2}.
$$
Since $n \geqslant 2$ and $q \geqslant 6$, we have $n^{q/2}/(n^{q/2} - q/2) \leqslant 2$.  Thus
$$
\frac{w(n,p^{m-i}k)^{p^s}}{w(n,p^{m-i+s}k)} \leqslant
2(p^{m-i}k)^{-p^s} p^{m-i} kp^s = \frac{2p^s}{(p^{m-i}k)^{p^s-1}}.
$$
This gives the required result.
\end{proof}

\begin{cor} \label{C:<=1}
Suppose that $m \geqslant 1$.
\begin{enumerate}
  \item[(i)] For $i = 2,3, \dotsc, m-1$, we have $a_i/a_{i-1} \leqslant 1$.
  \item[(ii)] We have $a_1 \leqslant 2/(p^{m-1}k)^{p-1}$ and $a_m \leqslant 2/k^{p^m-1}$.
\end{enumerate}
\end{cor}

\begin{proof} \hfill
\begin{enumerate}
  \item[(i)] By Lemma \ref{L:<=1} with $s=1$, we have
  $$
  \frac{a_i}{a_{i-1}} \leqslant p^{-1}\left( \frac{2p}{(p^{m-i}k)^{p-1}} \right)^{p^{i-1}} \leqslant \left( \frac{2}{p^{(m-i)(p-1)}} \right)^{p^i-1} \leqslant 1.
  $$
  \item[(ii)] By Lemma \ref{L:<=1} with $i=s=1$, we have
  $$
  a_1 = \frac{a_1}{a_0} \leqslant p^{-1} \frac{2p}{(p^{m-1}k)^{p-1}} = \frac{2}{(p^{m-1}k)^{p-1}}.
  $$
  Similarly, with $i = s = m$,
  $$
    a_m = \frac{a_m}{a_0} \leqslant p^{-m} \frac{2p^m}{k^{p^m-1}} = \frac{2}{k^{p^m-1}}.
$$
\end{enumerate}
\end{proof}

\begin{lem} \label{L:>=}
We have
$$
\frac{w(n^{p^m}, k)}{p^m w(n,p^m k)} \geqslant 1 - \frac{k}{2n^{p^m k/2}}.
$$
\end{lem}

\begin{proof}
By Lemma \ref{L:maininequality}, we have
$$
\frac{w(n^{p^m}, k)}{p^m w(n,p^m k)} \geqslant \frac{n^{p^m k}/k - n^{p^m k/2}/2}{p^m n^{p^mk}/(p^m k)} = 1 - \frac{k}{2n^{p^m k/2}}.
$$
\end{proof}

\begin{proof}[Proof of Theorem \ref{T:main1}]
Suppose that $m \geqslant 1$.  Recall that $b_r \leqslant 1$ for all $r$.  By \eqref{E:eq2} and Corollary \ref{C:<=1}(i), we have
$$
1\geqslant b_{p^m k} \geqslant w(n^{p^m}, k)/(p^m w(n,p^m k)) - (m-1)a_1 - a_m.
$$
Hence, by Corollary \ref{C:<=1}(ii) and Lemma \ref{L:>=},
$$
1\geqslant b_{p^mk} \geqslant 1 - \frac{k}{2n^{p^mk/2}} - \frac{2(m-1)}{(p^{m-1}k)^{p-1}} - \frac{2}{k^{p^m-1}}.
$$
Also, for $m = 0$, we have $b_k = w(n,k)/w(n,k) = 1$.  Thus $b_r \to 1$ as $r \to \infty$ with $r \in \mathcal{A}$.
\end{proof}

\section{Lie modules}\label{lie module section}

In this section we shall prove Theorem \ref{T:main2}.

In order to prove the theorem we shall apply the Schur functor to the modules appearing in a special case of the isomorphism \eqref{E:decompthm}. Since the Schur functor, in its usual form, requires the field $F$ to be infinite, we begin by showing that it is enough to prove Theorem \ref{T:main2} in the case where $F$ is infinite.

Let $G$ be a finite group and let $M$ be a finite-dimensional $FG$-module.  We write $M^{\max}$ for a maximal projective submodule of $M$.  Thus $M^{\max}$ is a maximal projective direct summand of $M$ and is uniquely determined up to isomorphism.  Let $K$ be an extension field of $F$.  Then $K \otimes M$ and $K \otimes M^{\max}$ are $KG$-modules.

\begin{lem} \label{L:fieldex}
Suppose that $F$ is a splitting field for $G$.  Then $K \otimes M^{\max} \cong (K \otimes M)^{\max}$.
\end{lem}

\begin{proof}
For modules $U$ and $V$, we write $U \mid V$ to denote that $U$ is isomorphic to a direct summand of $V$.  Clearly $K \otimes M^{\max}$ is projective and $K \otimes M^{\max} \mid K \otimes M$.  Thus $K \otimes M^{\max} \mid (K \otimes M)^{\max}$.  Since $F$ and $K$ are splitting fields for $G$, it follows from \cite[Theorem 10.18]{BHNB} that if $P$ is a projective indecomposable $FG$-module then $K \otimes P$ is a projective indecomposable $KG$-module and, furthermore, every projective indecomposable $KG$-module is isomorphic to some such $K \otimes P$.  Hence there is a projective $FG$-module $Q$ such that $K \otimes Q \cong (K \otimes M)^{\max}$.  In particular, $K \otimes Q \mid K \otimes M$ and it follows, by \cite[Theorem 1.21]{BHNB}, that $Q \mid M$.  Thus $Q \mid M^{\max}$ and so $(K \otimes M)^{\max} \mid K \otimes M^{\max}$.  The result follows.
\end{proof}

In Section \ref{S:intro} we defined the $F\sym{r}$-module $\Lie(r)$ for each positive integer $r$.  Let $\Lie_K(r)$ denote the $K\sym{r}$-module defined in the same way over $K$.  Thus $K \otimes \Lie(r) \cong \Lie_K(r)$.  As in Section \ref{S:intro}, let $\Liemax(r)$ be a maximal projective submodule of $\Lie(r)$.  It is well-known that every field is a splitting field for $\sym{r}$ (see \cite[Theorem 11.5]{GJ}).  Thus, by Lemma \ref{L:fieldex}, $K \otimes \Liemax(r)$ is isomorphic to a maximal projective submodule of $\Lie_K(r)$.  Consequently, if Theorem \ref{T:main2} holds over $K$, it holds over $F$.  Therefore it suffices to prove Theorem \ref{T:main2} in the case where $F$ is infinite.  Hence, from now on, we take $F$ to be infinite.

Suppose that $n \geqslant r$ and let $E_n$ be the $n$-dimensional natural module for the general linear group $\GL_n(F)$ with standard basis $\{e_1,\dotsc, e_n \}$.  We identify $\sym{r}$ with the subgroup of $\GL_n(F)$ consisting of all those elements which permute $\{ e_1, \dotsc, e_r\}$ and fix $e_{r+1}, \dotsc, e_n$ pointwise.

Let $\Lambda(n,r)$ be the set of all $n$-tuples $(\alpha_1,\dotsc, \alpha_n)$ of non-negative integers such that $\alpha_1 + \dotsb + \alpha_n = r$.  The elements of $\Lambda(n,r)$ are called weights.  For $t_1,\dotsc, t_n \in F^{\times} = F \setminus \{0\}$ let $d(t_1,\dotsc, t_n)$ denote the diagonal matrix in $\GL_n(F)$ with entries $t_1, \dotsc, t_n$ down the diagonal.

Let $M$ be a finite-dimensional $F\GL_n(F)$-module such that $M$ is a homogeneous polynomial module of degree $r$ (see \cite{G}).  For $\alpha = (\alpha_1,\dotsc, \alpha_n) \in \Lambda(n,r)$ the $\alpha$-weight space of $M$ is denoted by $M^{\alpha}$ and is the subspace of $M$ defined by
$$M^{\alpha} = \{ u \in M \mid u \cdot d(t_1,\dotsc, t_n) = t_1^{\alpha_1} \dotsm t_n^{\alpha_n} u,\,\, \forall t_1,\dotsc, t_n \in F^{\times} \}.
$$
By \cite[(3.2c)]{G}, we have
\begin{equation} \label{E:wtspace}
  M = \bigoplus_{\alpha \in \Lambda(n,r)} M^{\alpha}.
\end{equation}

Let $\omega_r \in \Lambda(n,r)$ be defined by $\omega_r = (1,\dotsc, 1, 0, \dotsc, 0)$ where there are $r$ entries equal to 1 and $n-r$ entries equal to 0.  For $\sigma \in \sym{r} \leqslant \GL_n(F)$ we have
$$
\sigma d(t_1,\dotsc, t_n) \sigma^{-1} = d(t_{1\sigma},\dotsc, t_{r\sigma}, t_{r+1},\dotsc, t_n).
$$
It follows that $M^{\omega_r}$ is invariant under the restriction to $\sym{r}$ of the action of $\GL_n(F)$ on $M$.  Hence $M^{\omega_r}$ is an $F\sym{r}$-module.

Let $\mathrm{mod}_F(n,r)$ be the class of all finite-dimensional homogeneous polynomial $F\GL_n(F)$-modules of degree $r$ and let $\mathrm{mod}(F\sym{r})$ be the class of all finite-dimensional $F\sym{r}$-modules.  Let
$$
f_r : \mathrm{mod}_F(n,r) \to \mathrm{mod}(F\sym{r})
$$
be the map defined by $f_r(M) = M^{\omega_r}$ for all $M \in \mathrm{mod}_F(n,r)$.  This map $f_r$ is called the `Schur functor'.  (It is an exact functor between the two module categories: see \cite[Chapter 6]{G}.)  It is easy to verify that
\begin{equation} \label{E:directsum}
f_r(M \oplus N) = f_r(M) \oplus f_r(N)
\end{equation}
for all $M, N \in \mathrm{mod}_F(n,r)$.  Furthermore, it is clear from the definition of the Lie module in Section \ref{S:intro} that
$$
f_r(L^r(E_n)) = \Lie(r),
$$ for all $n$ such that $n\geqslant r$.

Let $r_1,\dotsc, r_l$ be positive integers such that $r_1 +\dotsb + r_l = r$.  We may regard $\sym{r}$ as the group of all permutations of a set of cardinality $r$ written as the disjoint union of sets of cardinalities $r_1,\dotsc, r_l$.  Thus we may regard $\sym{r_1}\times \dotsb \times \sym{r_l}$ as a subgroup of $\sym{r}$.  The following lemma is an immediate consequence of a slightly stronger version given in \cite{SDKE}.

\begin{lem}[{\cite[\S2.5, Lemma]{SDKE}}] \label{L:DE}
Let $n \geqslant r$ and let $r_1,\dotsc, r_l$ be positive integers such that $r_1+ \dotsb + r_l = r$.  For $i=1,\dotsc, l$, let $M_i \in \mathrm{mod}_F(n,r_i)$.  Then
$$f_r(M_1 \otimes \dotsb \otimes M_l) \cong \Ind_{\sym{r_1}\times \dotsb \times \sym{r_l}}^{\sym{r}} ( f_{r_1}(M_1)\boxtimes \dotsb \boxtimes f_{r_l}(M_l)).
$$
\end{lem}

\begin{lem} \label{L:dim}
Let $n \geqslant qk$ where $q$ and $k$ are positive integers.  Then
$$
\dim f_{qk}(L^k(E_n^{\otimes q})) = (qk)!/k.
$$
\end{lem}

\begin{proof}
  Let $I(n,q)$ be the set of all ordered $q$-tuples $(i_1,\dotsc, i_q)$ where $i_1,\dotsc, i_q \in \{1,\dotsc, n\}$.  For $\theta \in I(n,q)$, where $\theta = (i_1,\dotsc, i_q)$, write
  $$e_{\theta} = e_{i_1} \otimes \dotsb \otimes e_{i_q} \in E_n^{\otimes q}.
  $$
  Thus $\{ e_{\theta} \mid \theta \in I(n,q) \}$ is a basis for $E_n^{\otimes q}$.

  The $F\GL_n(F)$-module $L^k(E_n^{\otimes q})$ is spanned as a vector space by the elements $[e_{\theta_1},\dotsc, e_{\theta_k}]$ with $\theta_1,\dotsc, \theta_k \in I(n,q)$.  It is easily verified that each of these elements belongs to some weight space of $L^k(E_n^{\otimes q})$.  Thus, by \eqref{E:wtspace}, $(L^k(E_n^{\otimes q}))^{\omega_{qk}}$ is spanned by those elements $[e_{\theta_1},\dotsc, e_{\theta_k}]$ which belong to it.  These are the elements $[e_{\theta_1},\dotsc, e_{\theta_k}]$ such that each of $e_1,\dotsc, e_{qk}$ occurs once and only once among the tensor factors of $e_{\theta_1},\dotsc, e_{\theta_k}$.

  Let $\Phi$ be the set of all $k$-element subsets $\{ \theta_1,\dotsc, \theta_k\}$ of $I(n,q)$ such that each $q$-tuple $\theta_j$ ($1 \leqslant j \leqslant k$) has $q$ distinct entries and $\{1,\dotsc, qk\}$ is the disjoint union of these $k$ sets of entries.  For each $\phi \in \Phi$ let $W_{\phi}$ be the subspace of $L^k(E_n^{\otimes q})$ spanned by the elements $[e_{\theta_1},\dotsc, e_{\theta_k}]$ where $\{\theta_1,\dotsc, \theta_k\} = \phi$.  It follows that
  $$
  f_{qk}(L^k(E_n^{\otimes q})) = (L^k(E_n^{\otimes q}))^{\omega_{qk}} = \bigoplus_{\phi \in \Phi} W_{\phi}.
  $$
  It is easily verified that $|\Phi| = (qk)!/k!$.  Also, by \cite[Theorem 5.11]{WMAKDS}, $\dim W_{\phi} = (k-1)!$ for all $\phi \in \Phi$. Thus
  $$
  \dim f_{qk}(L^k(E_n^{\otimes q})) = (qk)!(k-1)!/k! = (qk)!/k.
  $$
\end{proof}

We consider $B^r(V)$, as in Section \ref{lie power section}, in the special case where $G=\GL_r(F)$ and $V=E_r$ and we define $$C(r)=f_r(B^r(E_r)).$$ The following result will enable us to use $F\GL_n(F)$-modules for arbitrary values of $n$ with $n\geqslant r$.

\begin{lem}\label{Cr independent of n} For all $n\geqslant r$, we have $$C(r)\cong f_r(B^r(E_n)).$$
\end{lem}
\begin{proof} We regard $E_r$ as a subspace of $E_n$ with $\GL_r(F)\leqslant \GL_n(F)$ in the obvious way. Let $d_{n,r}:\module_F(n,r)\to \module_F(r,r)$ be the truncation map (see \cite[\S 6.5]{G}). For all $M\in \module_F(n,r)$, we have $$d_{n,r}(M)=\bigoplus_{\beta\in \Lambda(r,r)^\ast} M^\beta$$ where $\Lambda(r,r)^\ast$ is the subset of $\Lambda(n,r)$ consisting of those weights $(\beta_1,\ldots,\beta_n)$ such that $\beta_{r+1}=\cdots=\beta_n=0$ (see \cite[(6.5c)]{G}). In particular, $f_r(d_{n,r}(M))=f_r(M)$. It is a consequence of \cite[Theorem 4.2]{RBMS} that $d_{n,r}(B^r(E_n))=B^r(E_r)$. The result follows.
\end{proof}

By definition, $C(r)$ is a submodule of $\Lie(r)$. However, $B^r(E_r)$ is a direct summand of $E_r^{\otimes r}$ and, by \cite[(6.3d)]{G}, $f_r(E_r^{\otimes r})$ is a regular $F\sym{r}$-module. Thus, by \eqref{E:directsum}, $C(r)$ is a projective $F\sym{r}$-module, and so $\dim C(r) \leqslant \dim \Liemax(r)$.  Hence, in order to prove Theorem \ref{T:main2}, it suffices to prove that
$$
\lim_{\substack{r \to \infty \\ r \in \mathcal{A}}} \frac{\dim C(r)}{\dim \Lie(r)} = 1.
$$
For each $r \geqslant 1$, we define
\begin{equation} \label{E:c_r}
c_r = \frac{\dim C(r)}{\dim \Lie(r)} = \frac{\dim C(r)}{(r-1)!}.
\end{equation}
Thus $0 \leqslant c_r \leqslant 1$.  It suffices to prove that $c_r \to 1$ as $r \to \infty$ with $r \in \mathcal{A}$.

Suppose that $m$ is a non-negative integer and $k$ is a positive integer such that $p \nmid k$ and $k \geqslant 2$.  Take $n \geqslant r = p^mk$ and apply $f_r$ to the isomorphism \eqref{E:decompthm} with $V=E_n$.  By \eqref{E:directsum} we obtain
\begin{equation}\label{E:recurrenceB} \bigoplus_{i=0}^m p^{m-i}f_r(B^{p^{m-i}k}(E_n)^{\otimes p^i}) \cong f_r ( L^k(E_n^{\otimes p^m})).
\end{equation}
By Lemma \ref{L:DE}, Lemma \ref{Cr independent of n} and \eqref{E:c_r}, we have
$$
\dim f_r(B^{p^{m-i}k} (E_n)^{\otimes p^i}) = \frac{(p^mk)!}{((p^{m-i}k)!)^{p^i}} (\dim C(p^{m-i}k))^{p^i} =
\frac{(p^mk)!}{(p^{m-i}k)^{p^i}} (c_{p^{m-i}k})^{p^i}.
$$
Hence, by \eqref{E:recurrenceB} and Lemma \ref{L:dim}, we obtain
\begin{equation} \label{E:recurrenceLie}
\sum_{i=0}^m \frac{p^{m-i}(p^mk)!}{(p^{m-i}k)^{p^i}} (c_{p^{m-i}k})^{p^i} = \frac{(p^mk)!}{k}.
\end{equation}
For $i = 0,1,\dotsc, m$, define
$$
a_i' = (p^{m-i}k)^{-(p^i-1)}.$$
Then \eqref{E:recurrenceLie} may be written as
\begin{equation} \label{E:recurrenceLienew}
c_{p^mk} = 1 -
\sum_{i=1}^m a'_i (c_{p^{m-i}k})^{p^i}.
\end{equation}

\begin{lem} \label{L:equalityLie}
Suppose that $0 \leqslant s \leqslant i \leqslant m$.  Then
$$
\frac{a'_i}{a'_{i-s}} = p^{-s} \left(\frac{p^s}{(p^{m-i}k)^{p^s-1}} \right)^{p^{i-s}}.
$$
\end{lem}

\begin{proof}
We have
$$
\frac{a'_i}{a'_{i-s}} = \frac{(p^{m-i+s}k)^{p^{i-s}-1}}{(p^{m-i}k)^{p^i-1}} = \frac{(p^s)^{p^{i-s}-1}}{(p^{m-i}k)^{p^i-p^{i-s}}} = p^{-s} \left(\frac{p^s}{(p^{m-i}k)^{p^s-1}} \right)^{p^{i-s}}.
$$
\end{proof}

\begin{proof}[Proof of Theorem \ref{T:main2}]
Note the similarity between \eqref{E:recurrenceLienew} and \eqref{E:eq2}.  Lemma \ref{L:equalityLie} is stronger than Lemma \ref{L:<=1} and therefore gives an analogue of Corollary \ref{C:<=1}.  We do not need an analogue of Lemma \ref{L:>=} because the first term on the right-hand side of \eqref{E:recurrenceLienew} is 1.  Thus, by the argument of the proof of Theorem \ref{T:main1}, we obtain $c_r \to 1$ as $r \to \infty$ with $r \in \mathcal{A}$.
\end{proof}

\end{document}